\documentclass[12pt]{amsart}
\usepackage[english]{babel}
\usepackage{amsmath, amsthm, amssymb}
\usepackage{enumerate}
\usepackage{geometry}
\usepackage{hyperref}
\usepackage{cleveref}

\theoremstyle{plain}
\newtheorem{theorem}{Theorem}
\newtheorem*{maintheorem}{Theorem}

\newtheorem{lemma}[theorem]{Lemma}

\theoremstyle{definition}
\newtheorem{definition}[theorem]{Definition}
\newtheorem{remark}[theorem]{Remark}
\newtheorem{example}[theorem]{Example}

\let\er\R
\newcommand{\Q}{\mathbb{Q}}\let\kve\Q
\newcommand{\Z}{\mathbb{Z}}
\renewcommand{\O}{\mathcal{O}}

\DeclareMathOperator{\Tr}{Tr}\let\tr\Tr
\let\nm\Nm

\newcommand{\OK}{\mathcal O_K}
\let\epsilon\varepsilon\let\phi\varphi

\def\uv#1{``#1''}
\let\injto\hookrightarrow

\title{No proper generalized quadratic forms are universal over quadratic fields}
\author[]{Ond\v{r}ej Chwiedziuk}
\author[]{Mat\v{e}j Dole\v{z}\'{a}lek}
\author[]{Simona Hlavinkov\'a}
\author[]{Emma P\v{e}chou\v{c}kov\'{a}}
\author[]{Zden\v{e}k Pezlar}
\author[]{Om Prakash}
\author[]{Anna R\r{u}\v{z}i\v{c}kov\'{a}}
\author[]{Mikul\'{a}\v{s} Zindulka}
\address{Charles University, Faculty of Mathematics and Physics, Department of Algebra,
Sokolovsk\'{a} 83, 186 75 Praha 8, Czech Republic}
\email{ondrachwiedziuk@gmail.com}
\email{matej@gimli.ms.mff.cuni.cz}
\email{simonkahlavinkova@gmail.com}
\email{emma.pechouckova@gmail.com}
\email{zdendapezlar@seznam.cz}
\email{prakash@karlin.mff.cuni.cz}
\email{anna.ruzickova.13@gmail.com}
\email{mikulas.zindulka@matfyz.cuni.cz}

\thanks{
We acknowledge support by Czech Science Foundation (GACR) grant 21-00420M (O.P., M.Z.), Charles University programme PRIMUS/24/SCI/010 (O.P., M.Z.), Charles University Research Centre programme UNCE/24/SCI/022 (O.P.) and Charles University project GAUK No. 134824 (M.D., E.P., M.Z.); O.C., S.H., E.P., Z.P and A.R. were also supported by student faculty grants of the Faculty of Mathematics and Physics of Charles University.
}

\begin{document}

\begin{abstract}
We consider generalized quadratic forms over real quadratic number fields and prove, under a natural positive-definiteness condition, that a generalized quadratic form can only be universal if it contains a quadratic subform that is universal. We also construct an example illustrating that the positive-definiteness condition is necessary. 
\end{abstract}

\maketitle

\section{Introduction}

The study of quadratic forms furnishes much of the development and history of number theory. The focus is usually on identifying which numbers a given quadratic form represents, or more particularly on finding or characterizing forms which represent every (positive) integer -- such forms are called \emph{universal}. One of the early examples of a universal quadratic form over the integers was Lagrange's four-square theorem, which states that any positive integer may be expressed as $x^2+y^2+z^2+w^2$ for some integers $x$, $y$, $z$, $w$. Much later, all universal positive definite quadratic forms over the rational integers were satisfyingly characterized by the $290$ theorem of Bhargava and Hanke \cite{bhargava-hanke}, which states that a positive definite quadratic form over the integers is universal if and only if it represents each of the $29$ so-called \emph{critical integers}
\[\begin{gathered}
1,\ 2,\ 3,\ 5,\ 6,\ 7,\ 10,\ 13,\ 14,\ 15,\ 17,\ 19,\ 21,\ 22,\ 23,\ 26,\\
29,\ 30,\ 31,\ 34,\ 35,\ 37,\ 42,\ 58,\ 93,\ 110,\ 145,\ 203,\ 290.
\end{gathered}\]
This allowed for a full enumeration of all universal forms in four variables -- there are exactly 6436 of them.

While this effectively answered many of the questions concerning quadratic forms over the integers, the scope had been broadened to study quadratic forms over numbers fields as well at this point. Usually, this is restricted to totally real number fields $K$. In these, an elements is said to be \emph{totally positive} if it remains positive in each of the embeddings $K\injto \er$. A quadratic form is said to be \emph{totally positive definite}, if it only takes totally positive values away from the origin, and subsequently it is \emph{universal}, if it represents every totally positive element of the ring of integers $\OK$ of $K$. Famous examples in this area include Siegel's result that the only totally real number fields where some sum of squares is universal are $\kve$ with $x^2+y^2+z^2+w^2$ and $\kve(\sqrt5)$ with $x^2+y^2+z^2$ \cite{siegel}, remarkably requiring fewer variables than over $\kve$.

Various authors have provided numerous insights into universal forms over number fields, though many more questions remain open. Over quadratic fields $\kve(\sqrt D)$ specifically, several strong results relating minimal ranks and other characteristics of universal forms to continued fraction expansions of $\sqrt D$ have been established by Blomer and Kala \cite{blomer-kala1, blomer-kala2, kala}.
Related to Siegel's result on the sum of three squares in $\kve(\sqrt5)$, Kitaoka's conjecture tantalisingly posits that only finitely many totally real number fields admit a universal quadratic form in three variables, but this is still an open problem. Weaker versions of the conjecture have been proven for quadratic fields by Chan, Kim and Raghavan \cite{chan-kim-raghavan} and for fields of an arbitrary fixed degree by Kala and Yatsyna \cite{kala-yatsyna}.
Overall, the so-called \emph{indecomposable elements} seem to be closely intertwined with universal forms. We refer the reader to a survey paper by Kala for further details \cite{kala-survey}. Recently, Chan and Oh \cite{chan-oh} proved a generalization of the 290 theorem: for every totally real number field, there is a finite \uv{criterion set}  such that a form is universal if and only if it represents all elements of this criterion set.

One possible approach towards gaining insight about universal quadratic forms over number fields is to study representation by a broader class of functions, where one allows the unique non-trivial automorphism $\tau:K\to K$ to be applied to variables. In other words, one studies homogeneous quadratic polynomials not only in some variables $z_1,\dots,z_r$ but also simultaneously in their conjugates $\tau(z_1),\dots,\tau(z_r)$. Any quadratic form is also a generalized quadratic form, so one of the benefits of this approach might be to cast results on quadratic forms into a broader context, to see which notions hold more generally and which appear to be specific to quadratic forms.

These generalized quadratic forms were introduced recently in the context of any Galois extension of $\kve$ by Browning, Pierce, and Schindler \cite{browning-pierce-schindler}, though in this paper, we consider them only over quadratic fields.
They also generalize, albeit only slightly, Hermitian forms, for which Kim, Kim, and Park \cite{kim-kim-park} proved an analogue of the 15 theorem of Conway and Schneeberger, itself a precursor to the 290 theorem of Bhargava and Hanke. 
Recently, we along with Romeo \cite{stunts} investigated those generalized quadratic forms which only take rational values and characterized quadratic fields which admit a binary or ternary generalized form that is universal in the sense of representing all positive integers.

In this paper, we will study generalized quadratic forms over quadratic fields without the restriction that they take only rational values. We will introduce a notion of an \emph{associated quadratic form} $Q$ to a generalized quadratic form $G$, which essentially treats each $z$ and $\tau(z)$ as separate variables while also crucially discarding those that are not used in $G$ (see \Cref{def:associated-quadform}). We will use totally
positive definiteness of $Q$ as a measure of \uv{positive-definiteness} of $G$, and under this notion of positive-definiteness, our main result will be:
\begin{maintheorem}
    Let $ G(z_1, z_2, \dots, z_r) $ be a totally positive definite integral generalized form in $r$ variables over a real quadratic field $K$. If $G$ is universal, then $G$ has a universal quadratic subform.
\end{maintheorem}
With \Cref{def:subform} giving the exact meaning of \uv{subform} for this theorem, we will prove it in \Cref{sec:main} as \Cref{thm:main}.
Very informally speaking, this means that $G$ cannot be universal in any new, interesting way. Crucially though, this notion of positive-definiteness for $G$ allows for it to be universal without containing a universal quadratic subform when the associated quadratic form of $G$ is only totally positive semidefinite, as we will illustrate in \Cref{sec:counterexample}.

Our proof of \Cref{thm:main} is specific to quadratic fields, leaving open the question of whether an analogous theorem holds in higher-degree fields as well. We speculate in \Cref{rmrk:higher-degree} that over some totally real Galois extension of $\kve$ of degree $d\geq 3$, a generalized form which would contain more than $1$ but not all $d$ of the conjugates of some variable $z$ might manage to be universal, thus yielding a counterexample to \Cref{thm:main} in higher-degree fields.

\section*{Acknowledgments}
This research project was accomplished at the Student Number Theory Seminar at Charles University. We thank V\'{i}t\v{e}zslav Kala for organizing the seminar and providing helpful feedback.

\section{Preliminaries}

Let $K=\kve(\sqrt D)$ be a real quadratic field and let $\OK$ be its ring of integers. We denote by $\tau$ the unique non-trivial automorphism
\begin{align*}
    \tau:K&\to K,\\
    a+b\sqrt D &\mapsto a-b\sqrt D.
\end{align*}
Subsequently, we denote the \emph{norm} $\nm(\alpha):=\alpha\tau(\alpha)$ and \emph{trace} $\tr(\alpha):=\alpha+\tau(\alpha)$. We say an $\alpha\in K$ is \emph{totally positive} if $\alpha>0$ and $\tau(\alpha)>0$; we then denote this by $\alpha\succ 0$. More generally, we say that $\alpha\succ\beta$ if $\alpha-\beta\succ0$, and analogously, we define the relation $\succeq$. The subsets of totally positive elements of $K$ and $\OK$ are $K^+$ and $\OK^+$ respectively. Further, we denote by $\OK^\times$ the set of units (invertible elements) of $\OK$.

By a \emph{quadratic form} over $K$ in $r$ variables, we mean an $r$-variable homogeneous quadratic polynomial
\[
    Q(x_1,\dots,x_r) = \sum_{1\leq i\leq j\leq r} a_{ij}x_ix_j
\]
with $a_{ij}\in K$. We say $Q$ is \emph{integral} if all $a_{ij}$ lie in $\OK$. We say $Q$ is \emph{totally positive definite} (or \emph{totally positive semidefinite}), if $Q(x_1,\dots,x_r)\succ 0$ (or $Q(x_1,\dots,x_r)\succeq0$) for all $(x_1,\dots,x_r)\in K^r\setminus\{(0,\dots,0)\}$. An integral $Q$ is said to \emph{represent} $\alpha\in\OK$, if there exist $\beta_1,\dots,\beta_r\in\OK$ such that $Q(\beta_1,\dots,\beta_r)=\alpha$; further, a totally positive semidefinite integral $Q$ is said to be \emph{universal} if it represents all elements of $\OK^+$.

Next, let us define generalized quadratic forms over $K$. Informally, these arise by allowing $\tau(z)$ (which in general is not a polynomial in $z$) to be used alongside every variable $z$ in a quadratic form. We will understand a \emph{generalized quadratic form} (or \emph{generalized form} for short) in $r$ variables over $K$ to be an expression of the form
\[
G(z_1,\dots,z_r) = G_0(z_1,\tau(z_1),\dots,z_r,\tau(z_r))
\]
where $G_0$ is a quadratic form in $2r$ variables over $K$. We will further call $G$ \emph{integral} if $G_0$ is integral. Let us say that an integral $G$ represents an element $\alpha\in\OK$ if there exist $\beta_1,\dots,\beta_r\in\OK$ such that $G(\beta_1,\dots,\beta_r)=\alpha$.

Although we could always refer to $G$ using the quadratic form $G_0$, this may introduce some redundancy when $z_i$ (or $\tau(z_i)$) does not appear in $G$, i.e. when there is no term in $G_0$ that would, after substituting $z_1,\tau(z_1),\dots,z_r,\tau(z_r)$, contain $z_i$ (or $\tau(z_i)$). This motivates the following:

\begin{definition}
    Let us say that $z$ is a \emph{proper variable} of $G$ if both $z$ and $\tau(z)$ appear in $G$. Further, let us say $G$ is \emph{quadratic} if it has no proper variables.
\end{definition}
When a generalized form is quadratic, each variable $z$ appears either only as $z$ or only as $\tau(z)$. Since $\tau$ restricts to a bijection on $\OK$, this does not change which elements are represented, which is what motivates us to disregard these \uv{cosmetic} applications of $\tau$ and treat a generalized form with no proper variables as essentially a quadratic form.

\begin{definition}\label{def:associated-quadform}
    Let $G(z_1,\dots,z_r)$ be a generalized form in $r$ variables and suppose without loss of generality that $z_1,\dots z_{\ell}$ are not proper variables while $z_{\ell+1},\dots,z_r$ are proper. We then define \emph{the associated quadratic} form of $G$ to be the quadratic form in $k=\ell+2(r-\ell)=2r-\ell$ variables satisfying
    \[
        G(z_1,\dots,z_k)=Q(z_1', \dots, z_{\ell}', z_{\ell+1}, \tau(z_{\ell+1}),\dots,z_k,\tau(z_k)), 
    \]
    where each $z_j'$ is either $z_j$ or $\tau(z_j)$ for $j \leq \ell$ depending on the form in which it appears in $G$.

    Let us say that $G$ is a \emph{totally positive definite} (or \emph{totally positive semidefinite)} generalized quadratic form if its associated $Q$ is a totally positive definite (or totally positive semidefinite) quadratic form.
\end{definition}
Informally speaking, $Q$ arises from $G$ by treating $z_i$ and $\tau(z_i)$ as wholly separate variables when both appear in $G$ and disregarding those $z_i$ or $\tau(z_i)$ which do not appear.

We will be interested mainly totally positive definite generalized forms $G$. Note that total positive-definiteness is \emph{not} equivalent to $G(z_1,\dots,z_r)\succ0$ for $(z_1,\dots,z_r)\in K^r\setminus\{(0,\dots,0)\}$ (see \Cref{sec:counterexample}). For a totally positive semidefinite generalized form $G$, let us say that $G$ is \emph{universal} if it represents all elements of $\OK^+$.

\begin{definition}
\label{def:subform}
Let us say that a generalized form $H(z_{1},\dots,z_{k})$ is a \emph{subform} of a generalized form $G(z_1,\dots,z_r)$ if $H$ arises from $G$ by setting some of the variables to zero, i.e.
\[
    H(z_1,\dots,z_k) = G(z_1,\dots,z_k, 0,\dots,0)
\]
after some reordering of the variables of $G$.
\end{definition}

\section{Finding a universal quadratic subform}
\label{sec:main}

\begin{lemma}
\label{lem:large-associate}
    Let $\alpha \in \mathcal{O}_K^+ $. For every constant $\delta>0$ there exists a unit $\epsilon\in\OK^\times$ such that $\beta:=\epsilon^2\alpha$ satisfies $\beta<\delta$.
\end{lemma}
\begin{proof}
Let $\varepsilon_0$ be a unit in $\OK$ satisfying $\epsilon_0>1$.
Letting $\epsilon=\epsilon_0^{-n}$ and choosing increasingly larger positive integers $n$, we observe that $\beta=\epsilon_0^{-2n}\alpha$ tends toward $0$ from above. Hence, for sufficiently large $n$, the desired inequality will be achieved.
\end{proof}

The following lemma has been proven in diverse formulations and generalizations on various occasions (e.g. \cite[Lemma 5.2]{cassels}, \cite[Lemma 4]{kala-yatsyna}, \cite[Proposition 5.1]{prakash}), but we redo a simple proof for the benefit of the reader here:
\begin{lemma}
\label{lem:quadform-lowerbound}
Let $Q(x_1,\dots,x_n)$ be a totally positive definite quadratic form in $n$ variables over $K$. Then there exists a rational $\delta=\delta(K,Q)>0$ such that
\[
    Q(x_1,\dots,x_n) \succeq \delta(x_1^2+\cdots+x_n^2).
\]
\end{lemma}
\begin{proof}
Let $Q$ be given by a symmetric matrix $A$. Being a symmetric real matrix, $A$ is orthogonally diagonalizable with all its eigenvalues real.
Further, since $Q$ is positive definite, these eigenvalues $\lambda_1\leq\cdots\leq\lambda_n$ are all positive. Letting now $(y_1,\dots,y_n)$ be the coordinates of $(x_1,\dots,x_n)$ in the corresponding orthonormal basis, we obtain
\[
Q(x_1,\dots,x_n) = \lambda_1y_1^2+\cdots+\lambda_ny_n^2 \geq \lambda_1(y_1^2+\cdots+y_n^2) = \lambda_1(x_1^2+\cdots+x_n^2).
\]
If we now choose $\delta$ as some rational number from the interval $(0,\lambda_1]$, we then obtain $Q(x_1,\dots,x_n)\geq\delta(x_1^2+\cdots+x_n^2)$ and consequently also
\[
\tau(Q(x_1,\dots,x_n)) = Q(\tau(x_1),\dots,\tau(x_n)) \geq \delta(\tau(x_1)^2+\cdots+\tau(x_n)^2) = \tau(\delta(x_1^2+\cdots+x_n^2)),
\]
hence $Q(x_1,\dots,x_n)\succeq \delta(x_1^2+\cdots+x_n^2)$.
\end{proof}

\begin{lemma}
\label{lem:genform-lowerbound}
Let $ G(z_1, z_2, \dots, z_r) $ be a totally positive definite generalized  form in $ r $ variables over $ K $. Then there exists a constant $\delta>0$ with the following property: whenever $(z_1,\dots,z_r)\in\OK^+$ such that some proper variable $z_i$ takes a non-zero value, it holds that $G(z_1,\dots,z_r)\geq \delta$.
\end{lemma}
\begin{proof}
For notational simplicity, let us write the associated quadratic form $Q$ as $Q(x_1,\dots,x_n)$, disregarding the correspondence between these $x_j$'s and the various $z_i$ and $\tau(z_i)$. By \Cref{lem:quadform-lowerbound}, we have 
$Q(x_1,\dots,x_n)\geq \delta(x_1^2+\cdots+x_n^2)$ for a certain $\delta$. Now, letting $z_i$ be a proper variable that takes a non-zero value, we see that $z_i^2+\tau(z_i)^2$ appears in $x_1^2+\cdots+x_n^2$. Since all the other terms in this sum of squares are non-negative, we then bound
\[
    G(z_1,\dots,z_r) \geq \delta(z_i^2+\tau(z_i)^2) = \delta\Tr(z_i^2).
\]
Since $z_i^2\in\OK^+$, we obtain $\Tr(z_i)\geq1$. Thus we arrive at $G(z_1,\dots,z_n)\geq \delta$.
\end{proof}

\begin{theorem}
\label{thm:main}
    Let $ G(z_1, z_2, \dots, z_r) $ be a totally positive definite integral generalized form in $r$ variables over a real quadratic field $K$. If $G$ is universal, then $G$ has a universal quadratic subform.
\end{theorem}
\begin{proof}
    Assume without loss of generality that each of $z_1, \dots, z_\ell$ is not a proper variable of $G$ while $z_{\ell+1}, \dots, z_r$ are proper. Then $\tilde Q:=G(z_1,\dots,z_\ell,0,\dots,0)$ is a quadratic subform of $G$ and we will show it is universal.
    
    For this, let an arbitrary $\alpha\in\OK^+$ be given. Let us consider the constant $\delta$ from \Cref{lem:genform-lowerbound} corresponding to $G$. By \Cref{lem:large-associate}, we may choose a unit $\epsilon$ such that $\beta:=\epsilon^2\alpha$ satisfies $\beta<\delta$.

    This $\beta$ is still totally positive, so $G$ represents it, let us say $\beta = G(z_1,\dots,z_r)$ for some $z_1,\dots,z_r\in\OK$. Because of the choice of $\delta$ form \Cref{lem:genform-lowerbound}, we see that all $z_{\ell+1},\dots,z_r$ must be zero. Hence $\beta$ is in fact represented by $G(z_1,\dots,z_\ell,0,\dots,0)$, i.e. by $\tilde Q$. The property of being represented by a quadratic form is preserved when multiplying by the square of a unit, so since $\tilde Q$ represents $\beta$, it also represents $\alpha$. Thus $\tilde Q$ is universal.
\end{proof}

\begin{remark}
\label{rmrk:higher-degree}
Let us comment on why this proof only works for quadratic fields. The notion of a generalized quadratic form $G$ may easily be formulated over Galois number fields $K$ of degree $d$ as homogeneous quadratic polynomials in $\tau_j(z_i)$, where $i=1,\dots,r$ indexes the variables and $j=1,\dots,d$ indexes the automorphisms $\tau_j:K\to K$. The adjacent notion of associated quadratic form would also easily generalize: we would treat each $\tau_j(z_i)$ as a separate variable and forget those that do not appear.

The issue would arise in \Cref{lem:genform-lowerbound}: essentially, our proof in the quadratic case says that, using the lower bound by a sum of squares, once both $z$ and $\tau(z)$ are present, we bound the symmetric expression $z^2+\tau(z)^2$ from below. Thus we get a dichotomy: either both conjugates of the variable appear and we bound from below, or only one conjugate appears, contributing to a quadratic subform. This dichotomy breaks for $d\geq3$ -- for example, in the cubic case, a very small (potentially arbitrarily small) element could be represented using a variable $z$ for which $\tau_1(z)$ and $\tau_2(z)$ are present in $G$, but $\tau_3(z)$ is not. Hence we would not have a symmetric expression in $z$ to bound from below, but we would still be left with two forms of $z$, thus not getting representation by a quadratic subform.
\end{remark}

\section{A counterexample in the semidefinite case}
\label{sec:counterexample}

The definition of total positive-definiteness of a generalized form $G$ may seem somewhat unnatural, because the condition on the associated quadratic form $Q$ does not easily translate to a condition on $G$ itself. In contrast, asking that $G$ (or equivalently $Q$) be totally positive semidefinite is equivalent to $G$ attaining only totally non-negative values: if
\[
    G(z_1,\dots,z_r) = G_0(z_1,\tau(z_1),\dots,z_r,\tau(z_r))
\]
for some quadratic form $G_0$ and we denote $\sigma(K):=\{(z,\tau(z))\mid z\in K\}$, the values of $G$ on $K^r$ are the values of $G_0$ on the set $\sigma(K)^r$.
This set is dense in $\er^{2r}$, hence if $G_0$ is non-negative on $\sigma(K)^r$, it is non-negative on the whole of $\er^{2r}$. Density of $\sigma(K)^r$ is essentially the argument of \cite[Lemma 3.2]{stunts} justifying that $G$ determines $G_0$ uniquely.

With total positivity instead of total non-negativity, this correspondence breaks: Even if $G$ attains only totally positive values outside of the origin, $G_0$ (nor $Q$) need not be totally positive definite. This is because $\sigma(K)^r$ may miss some non-trivial points in $K^{2r}$ with zero values. E.g. the generalized form $G(z)=(2z-\tau(z))^2$ attains only totally positive values away from the origin, because $2z-\tau(z)=0$ only for $z=0$, but $G_0(x,y)=Q(x,y)=(2x-y)^2$ is only totally positive semidefinite, not definite.

We will now illustrate that the stronger condition of total positive definiteness of $G$ is needed in \Cref{thm:main} and cannot be replaced by total positive semidefiniteness.

\begin{example}
    Let us construct a totally positive semidefinite generalized form $G$ that is universal but has no universal quadratic subform. This will show that total positive-definiteness cannot be weakened to semidefiniteness in \Cref{thm:main}.

    Let $K:=\Q(\sqrt2)$, then $\O_K = \Z[\sqrt2]$. It is well-known \cite{cohn} that the quadratic form
    \[
        S := x_1^2+x_2^2+x_3^2+x_4^2
    \]
    represents all totally positive elements of $\Z[2\sqrt2]$, i.e. those totally positive $a+b\sqrt2$ with even $b$. It is further known (e.g. by applying \cite[Theorem 3]{dress-scharlau}) that
    the only indecomposable elements of $\O_K^+$, up to multiplying by squares of units, are $1$ and $\beta:=2+\sqrt2$.

    Denoting now $\varepsilon:=3+2\sqrt2$ (this is the square of the fundamental unit $1+\sqrt2$) and labeling
    \[
        \ell(z) := 2z-\tau(z),\hskip4em g(z) := \ell(z)^2 = (2z-\tau(z))^2,
    \]
    we claim that 
    \[
        H := g(z_1) + \varepsilon g(z_2) + \varepsilon^2 g(z_3) + \varepsilon^3 g(z_4) + \beta g(z_5) + \varepsilon\beta g(z_6) + \varepsilon^2\beta g(z_7) + \varepsilon^3\beta g(z_8)
    \]
    makes $G: = S+H$ into a universal generalized form.

    First, notice that $\ell(a+b\sqrt2) = a+3b\sqrt2$, so an element of $\O_K$ is expressible by $\ell$, if and only if its irrational part is divisible by $3$. All totally positive units in $\OK$ are powers of $\varepsilon$, and in $\OK/3\OK$ we may calculate
    \[
        \varepsilon^{2n} \equiv (2\sqrt2)^{2n} \equiv 8^n\equiv \pm 1\pmod 3,
    \]
    hence $\varepsilon^{2n}$ is expressible by $\ell$. This in turn means $\varepsilon^{4n}$ is represented by $g$. Thus, since every indecomposable element in $\O_K^+$ is of the form $\varepsilon^m$ or $\varepsilon^m\beta$ and $g$ can represent every power of $\varepsilon^4$, it follows that $H$ represents every indecomposable element in $\O_K^+$.

    Now we prove that $G$ is universal. Let $\alpha\in\O_K^+$ be arbitrary. If it has an even irrational part, it is represented by $S$, hence by $G$ as well. Thus we may presume that $\alpha$ has an odd irrational part. Every element of $\O_K^+$ is a sum of indecomposables, so let us write
    \[
        \alpha = \eta_1+\cdots+\eta_s
    \]
    for some indecomposables $\eta_1,\dots,\eta_s$. One of the $\eta_i$'s has to have an odd irrational part, since otherwise the irrational part of $\alpha$ would be even; without loss of generality, let it be $\eta_1$. Then $\alpha-\eta_1$ is either $0$ or totally positive, because it is a sum of non-negatively many totally positive elements, and simultaneously it has an even irrational part. Thus $S$ represents $\alpha-\eta_1$. By the discussion above, $H$ represents $\eta_1$, so altogether $G=S+H$ represents $\alpha = (\alpha-\eta_1)+\eta_1$. Hence $G$ is universal.

    Now, $G$ only takes totally non-negative values, because it is just some sum of squares of elements of $\O_K$ weighted by totally positive elements, hence it is totally positive semidefinite. But any quadratic subform of $G$ would have to be a subform of $S$, since all the variables from $H$ are proper. But (a subform of) a sum of squares can only represent elements of $\Z[2\sqrt2]$, meaning it cannot be universal.
    
    Thus $G$ is totally positive semidefinite but does not have a universal quadratic subform.
\end{example}

\def\arxiv#1{\href{https://arxiv.org/abs/#1}{arXiv:#1}}

\end{document}